\documentclass[12pt, draftcls, onecolumn]{IEEEtran}
\IEEEoverridecommandlockouts
\usepackage{newcent}
\usepackage{amsmath}
\usepackage{amssymb}
\usepackage{psfrag,graphicx}
\usepackage{epsfig}
\usepackage{verbatim}

\graphicspath{{../../../../Figures/}}

\newtheorem{thm}{Theorem}

\newtheorem{ex}{Example}

\newtheorem{ass}{Assumption}

\title{Deterministic Team Problems with Signaling Incentive}

\author{Ather Gattami\thanks{
Ather Gattami is with the Automatic Control Laboratory, Electrical Engineering School,
KTH-Royal Institute of Technology, 100 44, Stockholm, Sweden. E-mail: gattami@kth.se}}

\begin{document}
\maketitle
\begin{abstract}
This paper considers linear quadratic team decision problems 
where the players in the team affect each other's information structure 
through their decisions. Whereas the stochastic version 
of the problem is well known to be complex with nonlinear optimal solutions 
that are hard to find, the deterministic counterpart is shown to be tractable. 
We show that under a mild assumption, where the weighting matrix on the controller
is chosen large enough, linear decisions are optimal and can be found efficiently by solving a semi-definite program.
\end{abstract}


\section*{Notation}
\begin{tabular}{ll}
$\mathbb{S}^n$ & The set of $n\times n$ symmetric matrices.\\
$\mathbb{S}^n_{+}$& The set of $n\times n$ symmetric positive\\
& semidefinite matrices.\\
$\mathbb{S}^n_{++}$& The set of $n\times n$ symmetric positive\\
& definite matrices.\\
$\mathcal{C}$& The set of functions
$\mu:\mathbb{R}^p\rightarrow \mathbb{R}^m$ with\\
& $\mu(y)=(\mu_1(y_1), \mu_2(y_2), ..., \mu_N(y_N))$,\\
& $\mu_i:\mathbb{R}^{p_i}\rightarrow \mathbb{R}^{m_i}$,
$\sum_{i} m_i=m$, $\sum_{i} p_i=p$.\\
$\mathbb{K}$ & $\{K\in \mathbb{R}^{m\times p}| K=\oplus \sum K_i, K_i\in \mathbb{R}^{m_i\times p_i}\}$\\
$A^{\dagger}$  & Denotes the pseudo-inverse of the\\ & square matrix $A$. \\
$A_{\perp}$      & Denotes the matrix with minimal number\\ 
			& of columns spanning the nullspace of $A$.\\
$A_i$& The $i$th block row of the matrix $A$.\\
$A_{ij}$& The block element of $A$ in position
$(i, j)$.\\
$\succeq$ & $A\succeq B$ $\Longleftrightarrow$ $A-B\in \mathbb{S}^n_{+}$.\\
$\succ$ & $A\succ B$ $\Longleftrightarrow$ $A-B\in \mathbb{S}^n_{++}$.\\
$\mathbf{Tr}$& $\mathbf{Tr}[A]$ is the trace of the matrix $A$.\\
$\mathcal{N}(m,X)$&  The set of Gaussian variables with\\
& mean $m$ and covariance $X$.
\end{tabular}
\newpage
\section{Introduction}

The team problem is an optimization problem, where
a number of decision makers (or players) make up a team, 
optimizing a common cost function with respect to some uncertainty
representing \textit{nature}. Each member of the team has 
limited information about the global state of nature. Furthermore,
the team members could have different pieces of information, which
makes the problem different from the one considered in classical 
optimization, where there is only one decision function that has
access to the entire information available about the state of nature.

Team problems seemed to possess certain properties that were considerably 
different from standard optimization, even for specific problem structures such as 
the optimization of a quadratic cost in the state of nature and
the decisions of the team members. In stochastic linear quadratic
decision theory, it was believed for a while that
certainty-equvalence holds between estimation and optimal decision with complete 
information, even for team problems. The certainty-equivalence principle can be
briefly explained as follows. First assume that every team member has
access to the information about the entire state of nature, and 
find the corresponding optimal decision for each member. 
Then, each member makes an estimate of 
the state of nature, which is in turn combined with the optimal decision
obtained from the full information assumption. It turns out that
this strategy does \textit{not} yield an optimal solution (see \cite{radner}).

A general solution to static stochastic quadratic team problems 
was presented by Radner \cite{radner}. Radner's result
gave hope that some related problems of dynamic nature could be solved 
using similar arguments. But in 1968, Witsenhausen \cite{witsenhausen:1968} 
showed in his well known paper that finding the optimal decision can be complex 
if the decision makers affect each other's information. Witsenhausen
considered a dynamic decision problem over two time steps to 
illustrate that difficulty. The dynamic problem can actually be
written as a static team problem:
\begin{equation*}
\begin{aligned}
\text{minimize }   & \mathbf{E}\hspace{1mm}\left\{k_0u_0^2+(x+u_0-u_1)^2\right\}\\ 
\text{subject to } & u_0=\mu_0(x),\hspace{2mm}
u_1=\mu_1(x+u_0+w),
\end{aligned}    
\end{equation*}
where $x$ and $w$ are Gaussian with zero mean and variance $X$ and $W$, respectively. 
Here, we have two decision makers, one corresponding to $u_0$, and the other to $u_1$.
Witsenhausen showed that the optimal decisions $\mu_0$ and $\mu_1$ 
are not linear because of the \textit{signaling}/\textit{coding incentive} of $u_0$. Decision maker $u_1$ 
measures $x+u_0+w$, and hence, its measurement is affected by $u_0$. 
Decision maker $u_0$ tries to \textit{encode} 
information about $x$ in its decision,
which makes the optimal strategy complex. 

The problem above is actually an information theoretic problem. 
To see this, consider the slightly modified problem 
\begin{equation*}
\begin{aligned}
\text{minimize }   & \mathbf{E}\hspace{1mm}(x-u_1)^2\\ 
\text{subject to } & u_0=\mu_0(x), \hspace{2mm} \mathbf{E}\hspace{1mm}u_0^2 \leq 1,  \hspace{2mm}
u_1=\mu_1(u_0+w) 
\end{aligned}    
\end{equation*}
The modification made is that we removed $u_0$ from the objective function, and instead added a constraint 
$ \mathbf{E}\hspace{1mm}u_0^2 \leq 1$ to make sure that it has a limited variance (of course we could set an 
arbitrary power limitation on the variance). The modified problem is exactly the Gaussian channel coding/decoding 
problem (see Figure \ref{teamchannel})! The optimal solution to Witsenhausens counterexample is still unknown. 
Even if we would restrict the optimization problem to the set of linear decisions, there is still no known polynomial-time 
algorithm to find optimal solutions. Another interesting counterexample was recently given in \cite{lipsa:2008}.

\begin{figure}
\begin{center}
\psfig{file=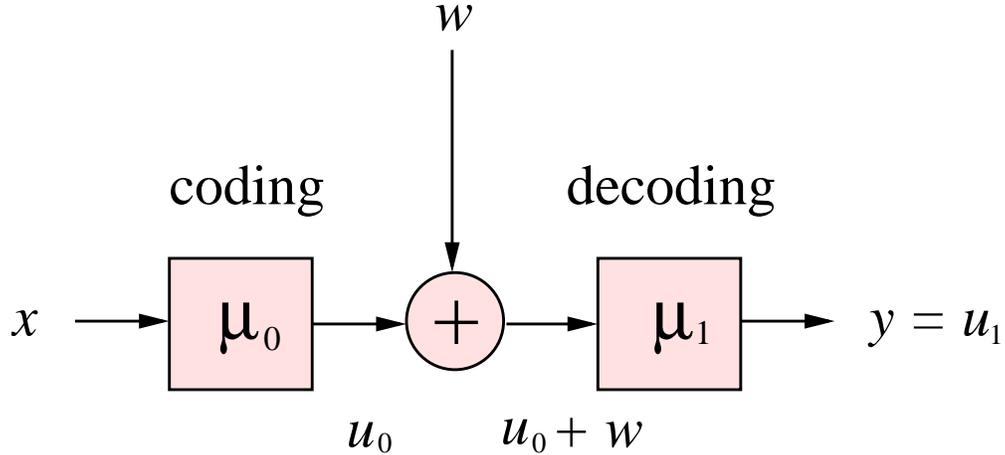,width=0.8\hsize}
\end{center}
\caption{Coding-decoding diagram over a Gaussian channel.}
\label{teamchannel}
\end{figure}

In this paper, we consider the problem of distributed decision
making with information constraints under linear quadratic
settings. For instance, information constraints appear naturally
when making decisions over networks. These problems can be
formulated as team problems. Early results considered static team
theory in stochastic settings \cite{marschak:1955}, \cite{radner},
\cite{ho:chu}. In \cite{didinsky:basar:1992}, the team problem
with two team members was solved. The solution cannot be easily
extended to more than two players since it uses the fact that the
two members have common information; a property that doesn't
necessarily hold for more than two players. \cite{didinsky:basar:1992} uses
the result to consider the two-player problem with one-step delayed measurement
sharing with the neighbors, which is a special case of the partially nested information 
structure, where there is no signaling incentive. Also, a nonlinear team
problem with two team members was considered in
\cite{bernhard:99}, where one of the team members is assumed to
have full information whereas the other member has only access to
partial information about the state of the world. Related team
problems with exponential cost criterion were considered in
\cite{krainak:82}. Optimizing team problems with respect to
\textit{affine} decisions in a minimax quadratic cost was shown to
be equivalent to stochastic team problems with exponential cost,
see \cite{fan:1994}. The connection is not clear when the
optimization is carried out over nonlinear decision functions.
In \cite{gattami:bob:rantzer}, a general solution was given for an 
arbitrary number of team members, where linear decision were shown
to be optimal and can be found by solving a linear matrix inequality.
In the deterministic version of Witsenhausen's counterexample, that is
minimizing the quadratic cost with respect to the worst case scenario of 
the state $x$ (instead of the assumption that $x$ is Gaussian), the linear 
decisions where shown to be optimal in \cite{rotkowitz:2006}. 

We will show that for static linear quadratic minimax team problems, 
where the players in the team affect each others information structure through their decisions, linear decisions are optimal  in general, and can be found by solving a linear matrix inequality.


\newpage
\section{Main Results}
\label{minimaxteam}

The deterministic problem considered is a quadratic game
between a team of players and nature. Each player has limited
information
that could be different from the other players in the team.
This game is formulated as a minimax problem, where the team is
the minimizer and nature is the maximizer.
We show that if there is a solution to the static minimax
team problem, then linear decisions are optimal, and
we show how to find a linear optimal solution by
solving a linear matrix inequality.

\section{Deterministic Team Problems with Signaling Incentive}
Consider the following team decision problem

\begin{equation}
\label{minimax_signaling}
\begin{aligned}
\inf_{\mu} \sup_{v\in \mathbb{R}^p, 0\neq w\in \mathbb{R}^q}\hspace{1mm} & 
\frac{L(w,u)}{\|w\|^2+\|v\|^2}\\
\text{subject to }\hspace{1mm} & y_i = \sum_{j=1}^ND_{ij}u_j+E_iw + v_i\\
                   & u_i = \mu_i(y_i)\\
                   & \text{for } i=1,..., N,
\end{aligned}
\end{equation}
where $u_i\in \mathbb{R}^{m_i}$ and $E_i\in \mathbb{R}^{p_i\times q}$, for $i=1, ..., N$, 

$L(w,u)$ is a quadratic cost given by
$$L(w,u)=
 \left[
\begin{matrix}
w\\
u
\end{matrix}
\right] ^T
\left[
\begin{matrix}
Q_{ww} & Q_{wu}\\
Q_{uw} & Q_{uu}
\end{matrix}
\right]
\left[
\begin{matrix}
w\\
u
\end{matrix}
\right] ,
$$
$Q_{uu}\in \mathbb{S}^m_{++}$, $m=m_1+\cdots + m_N$, and 
$$
\left[
\begin{matrix}
Q_{ww} &Q_{wu}\\
Q_{uw} & Q_{uu}
\end{matrix}
\right] \in \mathbb{S}^{m+n}_+.
$$

The players $u_1$,..., $u_N$ make up a
\textit{team}, which plays against \textit{nature} represented by
the vector $w$, using $\mu\in \mathcal{C}$. This problem is more complicated than the static team decision problem studied in \cite{gattami:bob:rantzer}, since it has the same flavour as that of the Witsenhausen counterexample that was presented in the introduction. We see that the measurement $y_i$ of decision maker $i$ could be affected by the other decision makers through the terms $D_{ij}u_j$, $j=1, ..., N$.

Note that we have the equality $y = Du+Ew+v$ which is equivalent to
$v = Du+Ew-y$. Using this substitution of variable, the team problem (\ref{minimax_signaling}) is equivalent to
\begin{equation}
\label{minimax_signaling2}
\begin{aligned}
\inf_{\mu\in\mathcal{C}} \sup_{y\in \mathbb{R}^p, 0\neq w\in \mathbb{R}^q}\hspace{1mm} & 
\frac{L(w,\mu(y))}{||D\mu(y)+Ew-y||^2+\|w\|^2}\\
\end{aligned}
\end{equation}

\begin{ass}
\label{ass2}
$$
\gamma^\star \leq \bar{\gamma} := 
\inf_{Du \neq 0}\hspace{1mm} 
\frac{u^TQ_{uu}u}{u^TD^TDu} .
$$

\end{ass}
\vspace{5mm}
\begin{thm}
\label{signaling_team}
Let  $\gamma^\star$ be the value of the game  (\ref{minimax_signaling}) and 
suppose that Assumption \ref{ass2} holds. 
Then the following statements hold:

\begin{itemize}
\item[($i$)]
There exist linear decisions $\mu_i(y_i)=K_iy_i$, $i=1,
..., N$, where the value $\gamma^\star $ is achieved.
\item[($ii$)]
If $\gamma^\star<\bar{\gamma}$, then for any $\gamma\in [\gamma^\star~,~\bar{\gamma})$,
a linear decision $Ky$ with $K\in \mathbb{K}$ that achieves $\gamma$ is obtained  by solving the 
linear matrix inequality
{\small
\begin{equation*}
\label{team_computation}
  \begin{aligned}
    \text{find }\hspace {2mm} & K\\
    \text{subject to }\hspace {2mm}   & K=\text{diag}(K_1, ..., K_N) \\
	&
    C = \left[
		\begin{matrix}
			I & 0
		\end{matrix}
	\right] \in \mathbb{R}^{p\times(p+q)}, \hspace{3mm}  Q_{uu}(\gamma )\in \mathbb{S}^{m\times m}
	\end{aligned}
\end{equation*}    

\begin{equation*}
\begin{aligned}
\left[
\begin{matrix}
Q_{xx}(\gamma)	& Q_{xu}(\gamma)\\   
Q_{ux}(\gamma)   & Q_{uu}(\gamma)
\end{matrix}
\right]  &=
	\left[
		\begin{matrix}
			 Q_{ww}    & 0    &Q_{wu}\\
			 0   		& 0 	& 0\\
			 Q_{uw}   	& 0  	& Q_{uu} 
		\end{matrix}
	\right]
	 - \gamma
	\left[
		\begin{matrix}
			  E^TE       & - E^T      & - E^TD\\
			 - E   		& I    		& - D\\
			 - D^TE  	& - D^T 	& D^TD

		\end{matrix}
	\right] 
\end{aligned}
\end{equation*}
$$
\left[\begin{matrix}
    Q_{xx}(\gamma) + Q_{xu}(\gamma)KC+C^TK^TQ_{ux}(\gamma) & C^TK^T\\
    KC & -Q_{uu}^{-1}(\gamma)
    \end{matrix}
\right] \preceq 0 ,
$$
}

\end{itemize}

\end{thm}
\begin{proof}

($i$) Note that 
$$y  = Du + Ew + v \Longleftrightarrow v = y - Du - Ew\Rightarrow$$
$$\Rightarrow \frac{L(w,u)}{||v||^2+\|w\|^2} = \frac{L(w,u)}{||y - Du - Ew||^2+\|w\|^2}.$$
Now introduce $x\in \mathbb{R}^{n}$, $n=p+q$, such that 
$$
x =	\left[
		\begin{matrix}
			w\\
			y
		\end{matrix}
	\right] ,
$$
and
\begin{equation}
\label{QR}
\begin{aligned}
Q  &=	
	\left[
		\begin{matrix}
			 Q_{ww}   & 0    & Q_{wu}\\
			 0   		& 0 	& 0\\
			Q_{uw}   	& 0  	& Q_{uu} 
		\end{matrix}
	\right] , \\
R  &=	
	\left[
		\begin{matrix}
			  E^TE       & - E^T      & - E^TD\\
			 - E   		& I    		& - D\\
			 - D^TE  	& - D^T 	& D^TD

		\end{matrix}
	\right].		
\end{aligned}
\end{equation}
Then,
\begin{equation*}
\begin{aligned}
J(x,u) &:= \left[
\begin{matrix}
x\\
u
\end{matrix}
\right]^T
Q
\left[
\begin{matrix}
x\\
u
\end{matrix}
\right] = L(w,u),
\\
F(x,u) &:= \left[
\begin{matrix}
x\\
u
\end{matrix}
\right]^T
R
\left[
\begin{matrix}
x\\
u
\end{matrix}
\right] = ||y - Du - Ew||^2+\|w\|^2,
\end{aligned}
\end{equation*}
and
$
y = Cx.
$
Hence, we have that
$$
\frac{L(w,u)}{||v||^2+\|w\|^2} = \frac{L(w,u)}{||y - Du - Ew||^2+\|w\|^2} = \frac{J(x,u)}{F(x,u)}.
$$
Then, for any $\gamma\in (\gamma^\star~,~\bar{\gamma})$, there exists a decision function $\mu\in\mathcal{C}$
such that
$$
J(x,\mu(Cx)) - \gamma F(x,\mu(Cx)) = 
\left[
\begin{matrix}
x\\
\mu(Cx)
\end{matrix}
\right] ^T
\left[
\begin{matrix}
Q_{xx}(\gamma)	& Q_{xu}(\gamma)\\   
Q_{ux}(\gamma)   & Q_{uu}(\gamma)
\end{matrix}
\right]
\left[
\begin{matrix}
x\\
\mu(Cx)
\end{matrix}
\right] 
\leq 0
$$
for all $x$. Under Assumption \ref{ass2}, we have that 
$$
Q_{uu}(\gamma) = Q_{uu} -\gamma D^TD\succ 0
$$
 for any $\gamma\in (\gamma^\star~,~\bar{\gamma}]$. Thus,
we can  apply Theorem 1 in \cite{gattami:bob:rantzer}, which implies that there must exist linear decisions that can achieve
any  $\gamma\in (\gamma^\star~,~\bar{\gamma}]$. By compactness, there must exist linear decisions that achieve $\gamma^\star$.\\

($ii$) Let $\mu(Cx) = KCx$ for  $K\in \mathbb{K}$. Then 

$$
\left[
\begin{matrix}
x\\
KCx
\end{matrix}
\right] ^T
\left[
\begin{matrix}
Q_{xx}(\gamma)	& Q_{xu}(\gamma)\\   
Q_{ux}(\gamma)   & Q_{uu}(\gamma)
\end{matrix}
\right]
\left[
\begin{matrix}
x\\
KCx
\end{matrix}
\right] 
\leq 0, ~ ~ \forall x
$$
$$
\Updownarrow
$$
$$
\left[
\begin{matrix}
I\\
KC
\end{matrix}
\right] ^T
\left[
\begin{matrix}
Q_{xx}(\gamma)	& Q_{xu}(\gamma)\\   
Q_{ux}(\gamma)   & Q_{uu}(\gamma)
\end{matrix}
\right]
\left[
\begin{matrix}
I\\
KC
\end{matrix}
\right] 
\preceq 0
$$
$$
\Updownarrow
$$

$$
Q_{xx}(\gamma)+  Q_{xu}(\gamma)KC+C^TK^TQ_{ux}(\gamma) + C^TK^T Q_{uu}(\gamma) KC\preceq 0
$$

$$
\Updownarrow
$$
$$
\left[\begin{matrix}
   Q_{xx}(\gamma) + Q_{xu}(\gamma)KC+C^TK^TQ_{ux}(\gamma) & C^TK^T\\
    KC & -Q_{uu}^{-1}(\gamma)
    \end{matrix}
\right] \preceq 0,
$$
and the proof is complete.
\end{proof}
\section{Linear Quadratic Control with Arbitrary Information Constraints}
Consider the dynamic team decision problem
\begin{equation}
\label{lq}
\begin{aligned}
\inf_{\mu} \sup_{w, v\neq 0}\hspace{2mm} & \frac{\sum_{k=1}^{M}\left[
\begin{matrix}
x(k)\\
u(k)
\end{matrix}
\right]^T
\left[
\begin{matrix}
Q_{xx} & Q_{xu}\\
Q_{ux} & Q_{uu}
\end{matrix}
\right]
\left[
\begin{matrix}
x(k)\\
u(k)
\end{matrix}
\right]}{\sum_{k=1}^M \|w(k)\|^2+\|v(k)\|^2}\\
\text{subject to } & x(k+1) = Ax(k)+Bu(k)+w(k)\\
                   & y_i(k) = C_ix(k)+v_i(k)\\
                   & u_i(k) = [\mu_k]_{i}(y_i(k)),  i=1,..., N.
\end{aligned}
\end{equation}
Now write $x(t)$ and $y(t)$ as
{\small
\begin{equation*}
\label{expansion}
\begin{aligned}
x(t) &= \sum_{k=1}^{t}A^{k}Bu(M-k)+\sum_{k=1}^{t}A^{k}w(M-k),\\
y_i(t) &= \sum_{k=1}^{t}C_iA^{k}Bu(M-k)+\sum_{k=1}^{t}C_iA^{k}w(M-k)+v_i(k).
\end{aligned}
\end{equation*}}
It is easy to see that the optimal control problem above is equivalent to a static team problem
of the form (\ref{minimax_signaling}). Thus, linear controllers are optimal under Assumption \ref{ass2}.

\begin{ex}
Consider the deterministic version of the Witsenhausen counterexample presented 
in the introduction:
\begin{equation*}
\begin{aligned}
\inf_{\mu_1,\mu_2} \hspace{1mm} & \gamma \\
\text{s. t. } &\frac{ {k^2}\mu_1^2(y_1) +({x_1}-\mu_2(y_2))^2}{x_0^2+w^2} \leq \gamma\\
&{y_1 = x_0}\\
&{x_1} = {x_0}+\mu_1(y_1)\\
&y_2 = x_1+w = {x_0}+\mu_1(y_1)+w
\end{aligned}
\end{equation*}
Substitue $x_0=y_1$, $x_1 = y_1 + \mu_1(y_1)$ and
$
w^2 = ({x_0}+\mu_1(y_1)-y_2)^2
$
 in the inequality
 $$
 \ {k^2}\mu_1^2(y_1) +({x_1}-\mu_2(y_2))^2 \leq \gamma ({x_0^2}+w^2).
 $$
Then, we get the equivalent problem
\begin{equation*}
\begin{aligned}
\inf_{\mu_1, \mu_2} \hspace{1mm} & \gamma \\
\text{s. t. } & {k^2}\mu_1^2(y_1)+(y_1+\mu_1(y_1)-\mu_2(y_2))^2 
\leq \gamma (y_1^2+(y_1+\mu_1(y_1)-y_2)^2)
\end{aligned}
\end{equation*}
Completing the squares gives the following equivalent inequality
$$
\left[
\begin{matrix}
y_1\\
y_2\\
\mu_1(y_1)\\
\mu_2(y_2)
\end{matrix}
\right]^T
\left[
\begin{matrix}
1-2\gamma & \gamma & 1-\gamma & -1\\
\gamma & -\gamma & \gamma & 0\\
1-\gamma & \gamma & {1+k^2 - \gamma } & {-1}\\
-1 & 0 & {-1} & {1}
\end{matrix}
\right]
\left[
\begin{matrix}
y_1\\
y_2\\
\mu_1(y_1)\\
\mu_2(y_2)
\end{matrix}
\right] \leq 0
$$
For $k^2=0.1$, we can search over $\gamma <\bar{\gamma} = k^2 = 0.1$, and we can use 
Theorem \ref{signaling_team} to deduce that linear decisions are optimal, and can be computed
by iteratively solving a linear matrix inequality, where the iterations are done with respect to $\gamma$. We find that
$${\gamma^\star \approx 0.0901},$$ 
$${\mu_1(y_1)=-0.9001 y_1},$$ 
$${\mu_2(y_2)=-0.0896 y_2}.$$ 
 
 For $k^2 = 1$, we iterate with respect to $\gamma < 1$, and we find 
 optimal linear decisions given by 
 \begin{equation*}
\begin{aligned}
{\mu_1(y_1)} & {=}  {-0.3856 y_1}\\
{\mu_2(y_2)} & {=}  {0.3840 y_2} 
\end{aligned}
\end{equation*}
$$\Downarrow$$
$$
\gamma^\star =  0.3820
$$

\end{ex}

\begin{ex}
Consider the deterministic counterpart of the multi-stage finite-horizon stochastic control problem that was considered in \cite{lipsa:2008}:
$$
\inf_{\mu_k:\mathbb{R}\rightarrow \mathbb{R}} \sup_{x_0, v_0, ..., v_{m-1}\in \mathbb{R}} \frac{(x_m-x_0)^2+ \sum_{k=0}^{m-2} \mu^2_k(y_k)}{x_0^2+v_0^2+\cdots+v_{m-1}^2}
$$
subject to the dynamics 
\begin{align}
x_{k+1} &= \mu_k(y_k) \nonumber\\
y_k &= x_k+v_k \nonumber .
\end{align}
It is easy to check that $\bar{\gamma}=1$ and 
$Q_{uu} -\gamma D^T D\succ 0$ for $\gamma<\bar{\gamma}$ (compare with Assumption \ref{ass2}) . Thus, 
linear decisions are optimal. This is compared to the stochastic version, where linear decisions where not optimal
for $m>2$.
\end{ex}
\section{Conclusions}
We have considered the static team problem in deterministic linear
quadratic settings where the team members may affect each others information. 
We have shown that decisions that are linear
in the observations are optimal and can be found by solving a
linear matrix inequality. 

For future work, it would be interesting to consider the case where the measurements are given by $y = Du +Ew + Fv$,
for an arbitrary matrix $F$.
\section{Acknowledgements}
The author is grateful to Professor Anders Rantzer, Professor Bo Bernhardsson, and the reviewers for 
valuable comments and suggestions.

This work is supported by the Swedish Research Council.

\bibliography{../../ref/mybib}

\end{document}